\documentclass[12pt]{extarticle}
\usepackage[utf8]{inputenc}
\usepackage[T1]{fontenc}
\usepackage[english]{babel}
\usepackage{amsthm}
\usepackage{amsmath}
\usepackage{amsfonts}
\usepackage{amssymb}
\usepackage{graphicx}
\usepackage{color}
\usepackage{todonotes}
\usepackage{hyperref}
\newtheorem{theorem}{Theorem}
\newtheorem{lemma}{Lemma}

\newtheorem{corollary}[theorem]{Corollary}
\theoremstyle{definition}

\theoremstyle{remark}
\newtheorem{remark}[theorem]{Remark}
\def\Z{\mathbb{Z}}
\def\R{\mathbb{R}}
\title{New identities for a sum of products of the Kummer functions}
\author{S.I.\:Kalmykov$^{\rm a,b}$ and D.B.\:Karp$^{\rm a,b}$\footnote{Corresponding author.  S.I.Kalmykov's email: \emph{sergeykalmykov@inbox.ru}  D.B.\:Karp's email: \emph{dimkrp@gmail.com}}
\\[10pt]
\small{\textit{$\phantom{1}^a$Far Eastern Federal University, Vladivostok, Russia}}\\\small{\textit{$\phantom{1}^b$Institute of Applied Mathematics, FEBRAS}}
}
\date{}
\begin{document}
\maketitle

\begin{abstract}
Recently, Feng, Kuznetsov and Yang discovered a very general reduction formula for a sum of products of the generalized hypergeometric functions  (J. \:Math. \:Anal. \:Appl. \:443(2016), 116--122). The main goal of this note is to present a generalization of a particular case of their identity when the generalized hypergeometric function is reduced to the Kummer function ${}_1F_1$. Our generalized formula contains an additional integer shift in the bottom parameter of the Kummer function.
The key ingredient of the proof is a summation formula for the Clausen series ${}_3F_{2}(1)$ with two integral parameter differences of opposite sign.
In the ultimate section of the paper we prove another formula for a particular product difference of the Kummer functions in terms of a linear combination of these functions.
\end{abstract}

\bigskip

Keywords: \emph{Kummer function, Clausen function, hypergeometric identity, summation theorem, duality relations for hypergeometric functions}

\bigskip

MSC2010: 33C15, 33C20

\bigskip

\paragraph{1. Introduction.} The history of transformation and reduction formulas for product sums of hypergeometric functions can be traced back to Euler, as his celebrated
identity
$$
{}_2F_1\!\left(\!\begin{array}{c}a, b\\c\end{array}\!\!\vline\,\,x\right)=(1-x)^{c-a-b}{}_2F_1\!\left(\!\begin{array}{c}c-a,c-b\\c\end{array}\!\!\vline\,\,x\right)
$$
immediately leads to the reduction formula
\begin{multline*}
{}_2F_1\!\left(\!\begin{array}{c}a, b\\c\end{array}\!\!\vline\,\,x\right){}_2F_1\!\left(\!\begin{array}{c}1-a, 1-b\\2-c\end{array}\!\!\vline\,\,x\right)
\\
-{}_2F_1\!\left(\!\begin{array}{c}c-a, c-b\\c\end{array}\!\!\vline\,\,x\right){}_2F_1\!\left(\!\begin{array}{c}1+a-c, 1+b-c\\2-c\end{array}\!\!\vline\,\,x\right)=0.
\end{multline*}
Throughout the paper the standard notation ${}_pF_{q}$ will be used to denote the generalized hypergeometric functions \cite[(2.1.2)]{AAR}. Later on, Gauss found the relation \cite[(6)]{Nesterenko1995}
\begin{multline*}
{}_2F_1\!\left(\!\begin{array}{c}a, b\\c\end{array}\!\!\vline\,\,x\right){}_2F_1\!\left(\!\begin{array}{c}-a, -b\\-c\end{array}\!\!\vline\,\,x\right)
\\
-x^2\frac{ab(c-a)(c-b)}{c^2(c^2-1)}{}_2F_1\!\left(\!\begin{array}{c}1-a,1-b\\2-c\end{array}\!\!\vline\,\,x\right){}_2F_1\!\left(\!\begin{array}{c}1+a,1+b, \\2+c\end{array}\!\!\vline\,\,x\right)=1
\end{multline*}
around the same time  when Legendre discovered an identity for a linear combination of three products of the complete elliptic integrals $K$ with complimentary arguments ($K$ is a particular cases of  the Gauss function ${}_2F_1$).  Legendre's formula was generalized by Elliott in 1904 to a linear combination of three products of the Gauss hypergeometric functions containing three independent parameters. Another two-parameter generalization of Legendre's identity was found in \cite{AQVV} which was soon thereafter generalized in \cite{BNPV} to an identity with four independent parameters covering both Elliott's formula and the formula from \cite{AQVV}.  Another identity for a linear combination of three products of the Gauss functions with coefficients depending quadratically on the argument and containing three independent parameters was given in \cite[Theorem~3.2]{GKLZ}.

An independent development started in 1931 with  a paper by Darling \cite{Darling} followed by Bailey's work \cite{Bailey} who found another method for proving Darling's identity and its generalizations. The identities discovered by Darling are probably the first reduction formulas for linear combinations of products of Clausen's  hypergeometric function ${}_3F_2$ (sometimes also called  Thomae's hypergeometric function).  It was probably Nesterenko who first obtained a reduction formula for a linear combination of products of generalized hypergeometric functions in his work on Hermite-Pad\'{e}  approximation \cite[Theorem~5]{Nesterenko1995}. Probably due to complicated notation his work remained largely unnoticed. It was referred to, however, in Gorelov's paper \cite{Gorelov2010}, where the author discovered another extension of Darling's formulas to general ${}_pF_{q}$ \cite[Corollary~1]{Gorelov2010}.

The years 2015-2016 marked an important breakthrough in the subject: Beukers and Jouhet published the paper \cite{BeukersJouhet}, where they presented a very general identity involving derivatives up to certain order of the generalized hypergeometric function ${}_{p+1}F_{p}$ and gave a purely algebraic proof based on the theory of general differential and difference modules.  Soon thereafter Feng, Kuznetsov and Young \cite{FKY} found another identity for the general ${}_pF_{q}$ function, involving more free parameters than Buekers-Jouhet formula and coinciding with it if the order of derivatives is set to zero. This particular case of Buekers-Jouhet's formula also coincides with Gorelov's formula \cite[(4)]{Gorelov2010} (for $p<q+1$ one also has to apply confluence to derive Gorelov's formula from that of  Buekers-Jouhet). The proof of Feng, Kuznetsov and Young is entirely different from that of the previous authors and is based on the so-called non-local derangement identity.  The ${}_1F_{1}$ particular case of their formula \cite[Theorem~1]{FKY} reads
\begin{multline*}
\frac{(1+a_1-b_1)_m}{a_1-a_2}{}_1F_1\!\left(\!\begin{array}{c}1+a_1-b_1+m\\1+a_1-a_2\end{array}\!\!\vline\,\,x\right)
{}_1F_1\!\left(\!\begin{array}{c}b_1-a_1\\1+a_2-a_1\end{array}\!\!\vline\,\,-x\right)
\\
+\frac{(1+a_2-b_1)_m}{a_2-a_1}{}_1F_1\!\left(\!\begin{array}{c}1+a_2-b_1+m\\1+a_2-a_1\end{array}\!\!\vline\,\,x\right)
{}_1F_1\!\left(\!\begin{array}{c}b_1-a_2\\1+a_1-a_2\end{array}\!\!\vline\,\,-x\right)=P_r(x),
\end{multline*}
where $P_r(x)$ is a polynomial of degree $|m|-1$, $m$ is an integer number.  Setting $\alpha=b_1-a_1$, $\gamma=1+a_2-a_1$, $t=-x$ and applying the Kummer transformation ${}_1F_1(a;b;x)=e^{x}{}_1F_1(b-a;b;-x)$ this formula may be cast into the form
\begin{multline*}
(1-\alpha)_m{}_1F_1\!\left(\!\begin{array}{c}\alpha\\\gamma\end{array}\!\!\vline\,\,t\right)
{}_1F_1\!\left(\!\begin{array}{c}1-\alpha+m\\2-\gamma\end{array}\!\!\vline\,\,-t\right)
\\
-(\gamma-\alpha)_{m}{}_1F_1\!\left(\!\begin{array}{c}\alpha-m\\\gamma\end{array}\!\!\vline\,\,t\right)
{}_1F_1\!\left(\!\begin{array}{c}1-\alpha\\2-\gamma\end{array}\!\!\vline\,\,-t\right)=P_r(t).
\end{multline*}
The main purpose of this note is to generalize the last formula by introducing the second integer shift into the bottom parameter.  The main result in presented in Theorem~\ref{th:1F1general}.  To prove this formula we needed a transformation formula for terminating Thomae-Clausen series ${}_3F_{2}(1)$ presented in Corollary~\ref{cr:3F2transformation}.  This transformation follows from a possibly new summation formula for non-terminating ${}_3F_{2}(1)$ given in Theorem~\ref{th:3F2transform}.
This formula treats the case when one top parameter is less than one bottom parameter by a nonnegative integer, while another top parameter is greater than the other bottom parameters by a (different) positive integer.  A number of closely related summation formulas attracted a lot of attention over last decade as witnessed  by numerous publications on this subject of which we only mention \cite{KimRathie,MillerParis,MillerSrivastava,ShpotSrivastava} and references therein.  In the ultimate section of the paper we present a linearization formula for another product difference of the Kummer function motivated by our previous research in log-concavity of series in inverse gamma functions \cite{KalmykovKarpITSF}.  Finally, let us mention that a similar generalization can be constructed for the general case of Feng, Kuznetsov and Young identity \cite[Theorem~1]{FKY}, which is our work in progress.

\paragraph{2. Main results.}
We will use the self-explanatory notation $\Z_{>0}$, $\Z_{\ge0}$ and $\Z_{<0}$ for the corresponding subsets of integers $\Z$ and $\R$ to denote the set of real numbers.
Throughout the paper the Pochhammer symbol $(a)_k$ is defined for both positive and negative $k$ as $\Gamma(a+k)/\Gamma(a)$, so that for $k\in\Z_{<0}$, $(a)_{k}=(-1)^{k}/(1-a)_{-k}$. Our first theorem is a summation formula for ${_{3}F_2}(1)$ which may be of independent interest.
\begin{theorem}\label{th:3F2transform}
Suppose $l\in\Z_{\ge0}$, $r\in\Z$ and $c<r-l$ if $c\in\R$ or $c\le{r-l}$ if $-c\in\Z_{\ge0}$.  Then for arbitrary $a,b\in\R$ the following identity holds true\emph{:}
\begin{multline}\label{eq:3F2transform}
{_{3}F_2}\left(\left.\!\!\!\begin{array}{c}a,b-r,c\\a-l,b+1\end{array}\right|1\!\right)
\\
=\frac{(-1)^l\Gamma(1-c)\Gamma(1+b)(a-c-l)_r(1-b)_r}{\Gamma(1+b-c)(1-a)_l(a-b)_{r-l}r!}
{_{3}F_2}\left(\left.\!\!\!\begin{array}{c}-r,1-a,c-b\\1-b,1-a+c+l-r\end{array}\right|1\!\right).
\end{multline}
In particular, both sides are zero if $r\in\Z_{<0}$.
\end{theorem}
\begin{proof} For $r<0$ the claim follows by $n=2$ case of \cite[(12)]{Karlsson1971}. Hence, in what follows we assume $r\ge0$.
We start with Karlsson-Minton summation formula (see \cite[(10)]{Karlsson1971} or \cite[(4.2)]{MillerSrivastava})
\begin{equation}\label{eq:KarlMinton}
{_{3}F_2}\left(\left.\!\!\!\begin{array}{c}\beta,\alpha+l,\gamma\\\beta+1,\alpha\end{array}\right|1\!\right)
=\frac{\Gamma(1-\gamma)\Gamma(\beta+1)(\alpha-\beta)_l}{\Gamma(1+\beta-\gamma)(\alpha)_l}
\end{equation}
valid for $\gamma<1-l$, where $l$ is a nonnegative integer, and another summation formula due to Karlsson \cite[(6)]{Karlsson1974}:
\begin{equation}\label{eq:Karlsson74}
{_{3}F_2}\left(\left.\!\!\!\begin{array}{c}\beta,\delta,\gamma\\\beta+r+1,\eta\end{array}\right|1\!\right)
=\frac{(\beta)_{r+1}}{r!}\sum\limits_{j=0}^{r}\frac{(-r)_j}{(\beta+j)j!}
{_{3}F_2}\left(\left.\!\!\!\begin{array}{c}\beta+j,\delta,\gamma\\\beta+j+1,\eta\end{array}\right|1\!\right).
\end{equation}
Combining these two formulas we get
\begin{multline}\label{eq:7.4.4.14}
{_{3}F_2}\left(\left.\!\!\!\begin{array}{c}\beta,\alpha+l,\gamma\\\beta+r+1,\alpha\end{array}\right|1\!\right)
=\frac{(\beta)_{r+1}}{r!}\sum\limits_{j=0}^{r}\frac{(-r)_j}{(\beta+j)j!}
{_{3}F_2}\left(\left.\!\!\!\begin{array}{c}\beta+j,\alpha+l,\gamma\\\beta+j+1,\alpha\end{array}\right|1\!\right)
\\
=\frac{(\beta)_{r+1}}{r!}\sum\limits_{j=0}^{r}\frac{(-r)_j}{(\beta+j)j!}
\frac{\Gamma(1-\gamma)\Gamma(\beta+j+1)(\alpha-\beta-j)_l}{\Gamma(1+\beta+j-\gamma)(\alpha)_l}
\\
=\frac{(\beta)_{r+1}(\alpha-\beta)_l\Gamma(1-\gamma)\Gamma(\beta)}{(\alpha)_l\Gamma(1+\beta-\gamma)r!}
{_{3}F_2}\left(\left.\!\!\!\begin{array}{c}-r,\beta,1+\beta-\alpha\\1+\beta-\gamma,1+\beta-\alpha-l\end{array}\right|1\!\right).
\end{multline}
Next, we apply the following transformation formula for terminating ${}_3F_2(1)$ \cite[Appendix, (IV)]{RJRJR}:
$$
{_{3}F_2}\left(\left.\!\!\!\begin{array}{c}-r,A,B\\D,E\end{array}\right|1\!\right)=\frac{(D-A)_r(B)_r}{(D)_r(E)_r}
{_{3}F_2}\left(\left.\!\!\!\begin{array}{c}-r,E-B,1-D-r\\1-B-r,1+A-D-r\end{array}\right|1\!\right).
$$
Setting $A=1+\beta-\alpha$, $B=\beta$, $D=1+\beta-\gamma$ and $E=1+\beta-\alpha-l$ yields
\begin{multline*}
{_{3}F_2}\left(\left.\!\!\!\begin{array}{c}-r,\beta,1+\beta-\alpha\\1+\beta-\gamma,1+\beta-\alpha-l\end{array}\right|1\!\right)
\\
=\frac{(\alpha-\gamma)_r(\beta)_r}{(1+\beta-\gamma)_r(1+\beta-\alpha-l)_r}
{_{3}F_2}\left(\left.\!\!\!\begin{array}{c}-r,1-\alpha-l,\gamma-\beta-r\\1-\beta-r,1-\alpha+\gamma-r\end{array}\right|1\!\right)
\end{multline*}
and, on substituting,
\begin{multline*}
{_{3}F_2}\left(\left.\!\!\!\begin{array}{c}\beta,\alpha+l,\gamma\\\beta+r+1,\alpha\end{array}\right|1\!\right)
\\
=\frac{(-1)^l(\beta)_{r+1}\Gamma(1-\gamma)\Gamma(\beta+r)(\alpha-\gamma)_r}{(\alpha)_l\Gamma(1+\beta-\gamma+r)(1+\beta-\alpha)_{r-l}r!}
{_{3}F_2}\left(\left.\!\!\!\begin{array}{c}-r,1-\alpha-l,\gamma-\beta-r\\1-\beta-r,1-\alpha+\gamma-r\end{array}\right|1\!\right).
\end{multline*}
Finally, replacing $\alpha+l\to{a}$, $\beta+r\to{b}$, $\gamma\to{c}$ we arrive at (\ref{eq:3F2transform}) after some elementary manipulations with gamma functions and Pochhammer symbols.
\end{proof}

\begin{remark}
Formula (\ref{eq:7.4.4.14})  is also given in \cite[(7.4.4.14)]{PBM3}.  As we could not trace its proof in
the literature we decided to give a complete derivation here.
\end{remark}

\begin{remark}
Certainly, formula (\ref{eq:3F2transform}) also holds for complex values of $a$, $b$ and $c$ by analytic continuation.
\end{remark}

\begin{corollary}\label{cr:3F2transformation}
If $m_1,m_2\in\Z$ and $\Z_{\ge0}\ni{n}\ge\max(m_1,m_2-m_1)$, then for arbitrary real $\alpha,\gamma$
\begin{multline}\label{eq:3F2transformation}
\frac{(1-\gamma)_{m_2}(\gamma-\alpha)_{m_1-m_2}(\alpha-m_1)_{n}}{(\gamma-m_2)_nn!}
{_{3}F_2}\left(\left.\!\!\!\begin{array}{c}-n,1-\alpha,1-\gamma+m_2-n\!\\2-\gamma,1-\alpha+m_1-n\end{array}\right|1\!\right)
\\
=\frac{(-1)^{n-m_2}(1-\alpha)_{n+m_1-m_2}}{(2-\gamma)_n(n-m_2)!}
{_{3}F_2}\left(\left.\!\!\!\begin{array}{c}-n+m_2,\alpha,\gamma-1-n\!\\\gamma,\alpha+m_2-m_1-n\end{array}\right|1\!\right).
\end{multline}
In particular, both sides vanish if $n<m_2$.
\end{corollary}
\begin{proof}
The claim follows from formula (\ref{eq:3F2transform}) on setting $\alpha=1-a$, $\gamma=1-b$, $n=-c$, $r=n-m_2$, $l=n-m_1$ and applying the easily verifiable identities
$$
\frac{(\gamma-m_2)_n}{(1-\gamma)_{m_2}}=(\gamma)_{n-m_2}(-1)^{m_2}
$$
and
$$
\frac{(1-\alpha)_{n+m_1-m_2}}{(\alpha-m_1)_{n}}=\frac{(-1)^{m_1}(1-\alpha+m_1)_{n-m_2}}{(\alpha)_{n-m_1}}.
$$
\end{proof}

Our main result in the following statement.

\begin{theorem}\label{th:1F1general}
Set $(x)_{+}=\max(0,x)$. For $m_1,m_2\in\mathbb{Z}$, $m_1^2+m_2^2>0$, the following identity holds
\begin{multline}\label{eq:1F1general}
\frac{(1-\alpha)_{m_1}t^{(m_2)_{+}}}{(2-\gamma)_{m_2}}
{}_{1}F_{1}\!\left(\!\begin{array}{l}\alpha\\\gamma\end{array}\!\!\vline\,\,t\right)
{}_{1}F_{1}\!\left(\!\begin{array}{l}1-\alpha+m_1\\2-\gamma+m_2\end{array}\!\!\vline\,\,-t\right)
\\
-(1-\gamma)_{m_2}(\gamma-\alpha)_{m_1-m_2}t^{(-m_2)_{+}}
{}_{1}F_{1}\!\left(\!\begin{array}{l}1-\alpha\\2-\gamma\end{array}\!\!\vline\,\,-t\right)
{}_{1}F_{1}\!\left(\!\begin{array}{l}\alpha-m_1\\\gamma-m_2\end{array}\!\!\vline\,\,t\right)
=P_{r}(t),
\end{multline}
where $P_r(t)$ is a polynomial of degree $r\le\max(|m_1|,|m_2-m_1|)-1$.  This polynomial is calculated by taking the first $r$ terms of the expression on the
left hand side.
\end{theorem}
\begin{remark}
Another way to express the conclusion of the above theorem is: the coefficients at $t^n$ in the expression on the left hand side of (\ref{eq:1F1general}) vanish for all $n\ge\max(|m_1|,|m_2-m_1|)$.
\end{remark}

\begin{proof}
The claim will be established once we show that the coefficients at $t^n$ in the two terms on the left hand side coincide for $n\ge\max(|m_1|,|m_2-m_1|)$.  These coefficients are
\begin{equation*}
L(n):=\frac{(1-\alpha)_{m_1}}{(2-\gamma)_{m_2}}\sum_{k=0}^{n-(m_2)_+}\frac{(\alpha)_k(1-\alpha+m_1)_{n-(m_2)_+-k}(-1)^{n-(m_2)_+-k}}{(\gamma)_k(2-\gamma+m_2)_{n-(m_2)_+-k}k!(n-(m_2)_+-k)!},
\end{equation*}
and
\begin{multline*}
R(n):=(1-\gamma)_{m_2}(\gamma-\alpha)_{m_1-m_2}
\\
\times\sum_{k=0}^{n-(-m_2)_+} \frac{(1-\alpha)_k(\alpha-m_1)_{n-(-m_2)_+-k}(-1)^k}{(2-\gamma)_k(\gamma-m_2)_{n-(-m_2)_+-k}k! (n-(-m_2)_+-k)!}.
\end{multline*}
Assume first that $m_2\ge0$. Then using $(c)_{n-k}=(-1)^k(c)_{n}/(1-c-n)_k$ we obtain
$$
L(n)=\frac{(-1)^{n-m_2}(1-\alpha)_{n+m_1-m_2}}{(2-\gamma)_n(n-m_2)!}
{_{3}F_2}\left(\left.\!\!\!\begin{array}{c}-n+m_2,\alpha,\gamma-1-n\!\\\gamma,\alpha+m_2-m_1-n\end{array}\right|1\!\right),
$$
$$
R(n)=\frac{(1-\gamma)_{m_2}(\gamma-\alpha)_{m_1-m_2}(\alpha-m_1)_{n}}{(\gamma-m_2)_nn!}
{_{3}F_2}\left(\left.\!\!\!\begin{array}{c}-n,1-\alpha,1-\gamma+m_2-n\!\\2-\gamma,1-\alpha+m_1-n\end{array}\right|1\!\right).
$$
Then, by Corollary~\ref{cr:3F2transformation} $L(n)=R(n)$ for $n\ge\max(m_1,m_2-m_1)$.

Next, take $m_2<0$.  Writing $m_2'=-m_2$ and $m_1'=-m_1$ we obtain:
\begin{multline*}
L(n)=\frac{(1-\alpha)_{-m_1'}}{(2-\gamma)_{-m_2'}}\sum_{k=0}^{n}\frac{(\alpha)_k(1-\alpha-m_1')_{n-k}(-1)^{n-k}}{(\gamma)_k(2-\gamma-m_2')_{n-k}k!(n-k)!}
\\
=\frac{(-1)^{m_1'-m_2'+n}(\gamma-1)_{m_2'}(1-\alpha-m_1')_{n}}{(\alpha)_{m_1'}(2-\gamma-m_2')_{n}n!}
\sum_{k=0}^{n}\frac{(-n)_k(\alpha)_k(\gamma-1+m_2'-n)_{k}}{(\alpha+m_1'-n)_{k}(\gamma)_kk!}
\\
=\frac{(-1)^{m_1'-m_2'+n}(\gamma-1)_{m_2'}(1-\alpha-m_1')_{n}}{(\alpha)_{m_1'}(2-\gamma-m_2')_{n}n!}
{_{3}F_2}\left(\left.\!\!\!\begin{array}{c}-n,\alpha,\gamma-1+m_2'-n\!\\\gamma,\alpha+m_1'-n\end{array}\right|1\!\right),
\end{multline*}
and
\begin{multline*}
R(n)=(1-\gamma)_{-m_2'}(\gamma-\alpha)_{m_2'-m_1'}
\\
\times\sum_{k=0}^{n-m_2'} \frac{(1-\alpha)_k(\alpha+m_1')_{n-m_2'-k}(-1)^k}{(2-\gamma)_k(\gamma+m_2')_{n-m_2'-k}k!(n-m_2'-k)!}
\\
=\frac{(-1)^{m_1'}(\alpha+m_1')_{n-m_2'}}{(1-\gamma+\alpha)_{m_1'-m_2'}(\gamma)_{m_2'}(\gamma+m_2')_{n-m_2'}(n-m_2')!}
\\
\times\sum_{k=0}^{n-m_2'} \frac{(-n+m_2')_k(1-\alpha)_k(1-\gamma-n)_{k}}{(2-\gamma)_k(1-\alpha+m_2'-m_1'-n)_{k}k!}
\\
=\frac{(-1)^{m_1'}(\alpha+m_1')_{n-m_2'}}{(1-\gamma+\alpha)_{m_1'-m_2'}(\gamma)_{m_2'}(\gamma+m_2')_{n-m_2'}(n-m_2')!}
\\
\times{_{3}F_2}\left(\left.\!\!\!\begin{array}{c}-n+m_2',1-\alpha,1-\gamma-n\!\\2-\gamma,1-\alpha+m_2-m_1-n\end{array}\right|1\!\right).
\end{multline*}
The equality $R(n)=L(n)$ is now equivalent to formula (\ref{eq:3F2transformation}) with the following replacements: $m_1'\to{m_1}$, $m_2'\to{m_2}$,
$\alpha\to1-\alpha$, $\gamma\to2-\gamma$.  Hence, $R(n)=L(n)$ is true for $n\ge\max(m_1',m_2'-m_1')=\max(-m_1,m_1-m_2)$.  Combining this condition with
$n\ge\max(m_1,m_2-m_1)$ we see that $R(n)=L(n)$ for all $n\ge\max(|m_1|,|m_2-m_1|)$ which completes the proof of the theorem.
\end{proof}

The following identities exemplify Theorem~\ref{th:1F1general}.

\medskip

\textbf{Example~1.} Take $m_1=m_2=1$. Then $r=0$ and $P_{0}(t)=$ the constant term on the left hand side of (\ref{eq:1F1general}). This yields
\begin{multline} \label{eq:polid1}
\frac{(\alpha-1)t}{(\gamma-1)(\gamma-2)}
{}_{1}F_{1}\!\left(\!\begin{array}{l}\alpha\\\gamma\end{array}\!\!\vline\,\,t\right)
{}_{1}F_{1}\!\left(\!\begin{array}{l}2-\alpha\\3-\gamma\end{array}\!\!\vline\,\,-t\right)
\\
+{}_{1}F_{1}\!\left(\!\begin{array}{l}1-\alpha\\2-\gamma\end{array}\!\!\vline\,\,-t\right)
{}_{1}F_{1}\!\left(\!\begin{array}{l}\alpha-1\\\gamma-1\end{array}\!\!\vline\,\,t\right)=1.
\end{multline}

\medskip

\textbf{Example~2.}
Take $m_1=0$, $m_2=-1$. Then $r=0$ and $P_{0}(t)=$ the constant term on the left hand side of (\ref{eq:1F1general}). This yields
\begin{multline} \label{eq:polid2}
{}_{1}F_{1}\!\left(\!\begin{array}{l}\alpha\\\gamma\end{array}\!\!\vline\,\,t\right)
{}_{1}F_{1}\!\left(\!\begin{array}{l}1-\alpha\\1-\gamma\end{array}\!\!\vline\,\,-t\right)
-\frac{(\gamma-\alpha)t}{\gamma(\gamma-1)}{}_{1}F_{1}\!\left(\!\begin{array}{l}1-\alpha\\2-\gamma\end{array}\!\!\vline\,\,-t\right)
{}_{1}F_{1}\!\left(\!\begin{array}{l}\alpha\\\gamma+1\end{array}\!\!\vline\,\,t\right)=1.
\end{multline}

\medskip

\textbf{Example~3.}
Take $m_1=1$, $m_2=2$. Then $r=0$ and $P_{0}(t)=$ the constant term on the left hand side of (\ref{eq:1F1general}). This yields
\begin{multline} \label{eq:polid3}
\frac{(1-\alpha)(\gamma-\alpha-1)t^2}{(\gamma-2)(1-\gamma)_{3}}{}_{1}F_{1}\!\left(\!\begin{array}{l}\alpha\\\gamma\end{array}\!\!\vline\,\,t\right)
{}_{1}F_{1}\!\left(\!\begin{array}{l}2-\alpha\\4-\gamma\end{array}\!\!\vline\,\,-t\right)
\\
+{}_{1}F_{1}\!\left(\!\begin{array}{l}1-\alpha\\2-\gamma\end{array}\!\!\vline\,\,-t\right)
{}_{1}F_{1}\!\left(\!\begin{array}{l}\alpha-1\\\gamma-2\end{array}\!\!\vline\,\,t\right)=1.
\end{multline}

\textbf{Example~4.}
Take $m_1=1$, $m_2=0$. Then $r=0$ and $P_{0}(t)=$ the constant term on the left hand side of (\ref{eq:1F1general}). This yields
\begin{multline} \label{eq:polid4}
\frac{1-\alpha}{1-\gamma}{}_{1}F_{1}\!\left(\!\begin{array}{l}\alpha\\\gamma\end{array}\!\!\vline\,\,t\right)
{}_{1}F_{1}\!\left(\!\begin{array}{l}2-\alpha\\2-\gamma\end{array}\!\!\vline\,\,-t\right)
\\
-\frac{\gamma-\alpha}{1-\gamma}{}_{1}F_{1}\!\left(\!\begin{array}{l}1-\alpha\\2-\gamma\end{array}\!\!\vline\,\,-t\right)
{}_{1}F_{1}\!\left(\!\begin{array}{l}\alpha-1\\\gamma\end{array}\!\!\vline\,\,t\right)=1.
\end{multline}

\textbf{Example~5.}
Take $m_1=-1$, $m_2=1$. Then $r=1$ and $P_{1}(t)=$ the constant term plus the linear term on the left hand side of (\ref{eq:1F1general}). This yields
\begin{multline} \label{eq:polid5}
\frac{(\gamma-\alpha-2)_{2}t}{(1-\gamma)_{2}}{}_{1}F_{1}\!\left(\!\begin{array}{l}\alpha\\\gamma\end{array}\!\!\vline\,\,t\right)
{}_{1}F_{1}\!\left(\!\begin{array}{l}-\alpha\\3-\gamma\end{array}\!\!\vline\,\,-t\right)
\\
+\alpha{}_{1}F_{1}\!\left(\!\begin{array}{l}1-\alpha\\2-\gamma\end{array}\!\!\vline\,\,-t\right)
{}_{1}F_{1}\!\left(\!\begin{array}{l}\alpha+1\\\gamma-1\end{array}\!\!\vline\,\,t\right)=\alpha+t.
\end{multline}

\textbf{Example~6.}
Take $m_1=2$, $m_2=4$. Then $r=1$ and $P_{1}(t)=$ the constant term plus the linear term on the left hand side of (\ref{eq:1F1general}). This yields
\begin{multline} \label{eq:polid6}
\frac{(\gamma-\alpha-2)_{2}(1-\alpha)_{2}t^4}{(1-\gamma)_5}{}_{1}F_{1}\!\left(\!\begin{array}{l}\alpha\\\gamma\end{array}\!\!\vline\,\,t\right)
{}_{1}F_{1}\!\left(\!\begin{array}{l}3-\alpha\\6-\gamma\end{array}\!\!\vline\,\,-t\right)
\\
-(2-\gamma)_{3}
{}_{1}F_{1}\!\left(\!\begin{array}{l}1-\alpha\\2-\gamma\end{array}\!\!\vline\,\,-t\right)
{}_{1}F_{1}\!\left(\!\begin{array}{l}\alpha-2\\\gamma-4\end{array}\!\!\vline\,\,t\right)=(\gamma-4)_{3}+(2\alpha-\gamma)(\gamma-3)t.
\end{multline}

\bigskip

\paragraph{3. A linearization identity.}
In our paper \cite{KalmykovKarpITSF}, while studying log-concavity of general power series with respect to a parameter contained in the argument of the reciprocal
gamma function, we proved the following lemma.
\begin{lemma}\label{lm:recgamma}
{\rm (\cite[Lemma 2.3]{KalmykovKarpITSF})} Let $m\in\Z_{\ge0}$. Then for all complex $\beta$ and $\mu$
\begin{multline}\label{eq:recgamma}
\sum_{k=0}^{m} \left(\frac{1}{\Gamma(k+\mu+1)\Gamma(m-k+\mu+\beta)}-\frac{1}{\Gamma(k+\mu)\Gamma(m-k+\mu+\beta+1)}\right)
\\
=\frac{(\mu+\beta)_{m+1}-(\mu)_{m+1}}{\Gamma(\mu+m+1)\Gamma(\mu+\beta+m+1)}.
\end{multline}
\end{lemma}
Computing the generating functions on both sides leads to the next statement.
\begin{theorem}\label{th:linearization}
For all complex values of the parameters and the argument the following identity holds\emph{:}
\begin{multline}\label{eq:1F1}
(\mu+\beta){}_1F_1(1;\mu+1;x) {}_1F_1(1;\mu+\beta;x)-\mu{}_1F_1(1;\mu;x) {}_1F_1(1;\mu+1+\beta;x)
\\
=(\mu+\beta){}_1F_1(1;\mu+1;x)- \mu {}_1F_1(1;\mu+1+\beta;x).
\end{multline}
\end{theorem}

\begin{proof}
Multiply both sides of  (\ref{eq:recgamma}) by $x^m$  and sum over $m\in\Z_{\ge0}$. The the left hand side takes the form
\begin{multline*}
\sum_{m=0}^{\infty}\frac{x^m}{\Gamma(m+\mu+1)}\sum_{m=0}^{\infty} \frac{x^m}{\Gamma(m+\mu+\beta)}- \sum_{m=0}^{\infty}\frac{x^m}{\Gamma(m+\mu)}\sum_{m=0}^{\infty} \frac{x^m}{\Gamma(m+\mu+1+\beta)}
\\
=\sum_{m=0}^{\infty}\frac{m!}{\Gamma(\mu+1)(\mu+1)_m}\frac{x^m}{m!}\sum_{m=0}^{\infty} \frac{m!}{\Gamma(\mu+\beta)(\mu+\beta)_m}\frac{x^m}{m!} \\- \sum_{m=0}^{\infty}\frac{m!}{\Gamma(\mu)(\mu)_m}\frac{x^m}{m!}\sum_{m=0}^{\infty} \frac{m!}{\Gamma(\mu+1+\beta)(\mu+1+\beta)_m}\frac{x^m}{m!}\\
=\frac{1}{\Gamma(\mu+1)\Gamma(\mu+\beta)}{}_1F_1(1;\mu+1;x){}_1F_1(1;\mu+\beta;x) \\-\frac{1}{\Gamma(\mu)\Gamma(\mu+1+\beta)}{}_1F_1(1;\mu;x) {}_1F_1(1;\mu+1+\beta;x).
\end{multline*}
Concerning the right hand side, we have
\begin{multline*}
\sum_{m=0}^{\infty} \frac{(\mu+\beta)_{m+1} x^m}{\Gamma(\mu+m+1)\Gamma(\mu+m+1+\beta)}-\sum_{m=0}^{\infty} \frac{(\mu)_{m+1} x^m}{\Gamma(\mu+m+1)\Gamma(\mu+m+1+\beta)} \\
=
\frac{1}{\Gamma(\mu+\beta)\Gamma(\mu+1)}\sum_{m=0}^{\infty} \frac{m!}{(\mu+1)_m} \frac{x^m}{m!}-\frac{1}{\Gamma(\mu)\Gamma(\mu+1+\beta)}\sum_{m=0}^{\infty} \frac{m!}{(\mu+1+\beta)_m}\frac{x^m}{m!}
\\
= \frac{1}{\Gamma(\mu+\beta)\Gamma(\mu+1)} {}_1F_1(1;\mu+1;x)-\frac{1}{\Gamma(\mu)\Gamma(\mu+1+\beta)} {}_1F_1(1;\mu+1+\beta;x).
\end{multline*}
Multiplying both sides by $\Gamma(\mu+1)\Gamma(\mu+1+\beta)$ we get (\ref{eq:1F1}).
\end{proof}
Another identity for a similar product difference of the Kummer functions has been found in \cite[Lemma~7]{KalmykovKarpJMAA}.

\medskip

\paragraph{4. Acknowledgements.}
We thank Professor Yurii A. Brychkov and Professor Richard B. Paris for useful communication regarding formula (\ref{eq:7.4.4.14}).
This research  has been supported by the Russian Science Foundation under project 14-11-00022.

%\bibliographystyle{amsplain}
%\bibliography{amsrefs}

\end{document}